\documentclass{amsart}
\usepackage[utf8]{inputenc}
\usepackage{amsmath}
\usepackage{amsthm}
\usepackage{amsfonts}
\usepackage{amssymb}
\usepackage{graphicx}
\usepackage{xcolor}
\usepackage{enumitem}

\newtheorem{theorem}{Theorem}

\newtheorem{lemma}{Lemma}
\newtheorem{proposition}{Proposition}

\newtheorem{remark}{Remark}
\newtheorem{definition}{Definition}
\newtheorem{example}{Example}
\usepackage[h]{esvect}

\newcommand{\dG}{\vv{G}}

\newcommand{\dQ}{\vv{Q}}

\makeatletter
\DeclareRobustCommand{\cev}[1]{%
  {\mathpalette\do@cev{#1}}%
}
\newcommand{\do@cev}[2]{%
  \vbox{\offinterlineskip
    \sbox\z@{$\m@th#1 x$}%
    \ialign{##\cr
      \hidewidth\reflectbox{$\m@th#1\vec{}\mkern4mu$}\hidewidth\cr
      \noalign{\kern-\ht\z@}
      $\m@th#1#2$\cr
    }%
  }%
}
\makeatother

\begin{document}

\title{On non-isomorphic biminimal pots realizing the cube}

\author[M.M. Ferrari]{Margherita Maria Ferrari}
\address{Department of Mathematics and Statistics, University of South Florida, Tampa, FL 33620, USA; currently at Department of Mathematics, University of Manitoba, Winnipeg, MB R3T 2N2, Canada}
\email{margherita.ferrari@umanitoba.ca}

\author[A. Pasotti]{Anita Pasotti}
\address{DICATAM - Sez. Matematica, Universit\`a degli Studi di Brescia, Via
Branze 43, I-25123 Brescia, Italy}
\email{anita.pasotti@unibs.it}

\author[T.Traetta]{Tommaso Traetta}
\address{DICATAM - Sez. Matematica, Universit\`a degli Studi di Brescia, Via
Branze 43, I-25123 Brescia, Italy}
\email{tommaso.traetta@unibs.it}

\begin{abstract}
In this paper, we disprove a conjecture recently proposed in [L. Almodóvar et al.,
https://arxiv.org/abs/2108.00035] on the non-existence of biminimal pots realizing the cube, namely pots with
the minimum number of tiles and the minimum number of bond-edge types.
In particular, we present two biminimal pots realizing the cube and show that these two pots are unique up to isomorphisms.
\end{abstract}

\maketitle

\section{Introduction}
In this paper, we study two new graph invariants, namely the minimum number of tile types and the minimum number of bond-edge types. These two parameters have been introduced to optimally build a target graph-like structure through DNA self-assembly processes using branched junction molecules~\cite{Ellis2014}. We refer the reader to~\cite{Ellis2019} for an overview of mathematical problems arising from DNA self-assembly processes. In this context, the minimum number of tile types corresponds to the smallest number of molecules required to build a target graph $G$; the minimum number of bond-edge types represents instead the smallest number of bonds needed to construct $G$. The increasing interest in these parameters is reflected in the growing body of work addressing their computations for certain classes of graphs under different experimental conditions
\cite{Almodovar21,Almodovar2021b,Almodovar2019,Bonvicini2020,Ellis2014,Ferrari2018,Griffin,Mattamira2020}. Except for~\cite{Ferrari2018}, in all these studies it is assumed that the molecules have ``flexible'' arms, so that no geometric property has to be encoded into the mathematical formalism.

Here, we focus on a conjecture about the assembly of the cube $Q$ posed in~\cite{Almodovar21}. In this recent paper, Almod\' ovar et al.~concentrated on determining an optimal set of molecules, called \emph{pot}, that realizes the cube $Q$ so that no graphs with order smaller than that of $Q$, or having the same order but not isomorphic to $Q$, can be realized. More precisely, the authors provide two distinct pots realizing the cube in this scenario: one pot utilizes the minimum number of tile types, which equals $6$, and the other utilizes the minimum number of bond-edge types, which is $5$. They conjecture that there is no pot achieving both the minimum number of tile types \emph{and} the minimum number of bond-edge types for the cube $Q$. In this work, we disprove this conjecture by presenting two \emph{biminimal} pots realizing the cube and having $6$ tile types and $5$ bond-edge types. Importantly, we show that these two pots are unique (up to isomorphisms). To this end, we employ edge-colorings and orientations of an undirected graph $G$ to describe the problem of computing the minimum number of tile types and the minimum number of bond-edge types of $G$. Note that these notions have already been employed to study the two quantities of interest; see~\cite{Ellis2014} and~\cite{Bonvicini2020} where, in the latter, this topic is strongly related to graph decompositions~\cite{BEZ} and some chromatic parameters. Nonetheless, we believe our setup facilitates the mathematical expositions of the theoretical results while offering another avenue to study the minimum number of tile types and the minimum number of bond-edge types of a graph $G$ (for instance, via signed graphs).

This paper is organized as follows. In Section~\ref{sec:prel}, we introduce the background definitions and basic properties used in this work. In Section~\ref{sec:res}, we recall known results related to the assembly of the cube. In Section~\ref{sec:pots}, we show that there are only two biminimal pots (up to isomorphisms) that realize the cube in the considered setting.


\section{Preliminaries}\label{sec:prel}
In this paper, graphs are connected and may have loops or multiple edges. To avoid ambiguity, we define a graph $G$ as a triple $(V(G), E(G), \mu)$ where
$V(G)$ is the set of vertices,
$E(G)$ is the set of edges, and $\mu: E(G)\rightarrow V(G)^{(2)}$ is a map,
where
\begin{enumerate}
\item $V(G)^{(2)}=V(G)\times V(G)$, and in this case $G$ is \emph{directed};
\item $V(G)^{(2)}$ is the set of (not necessarily distinct) unordered pairs of vertices of $G$, and in this case $G$ is \emph{undirected}.
\end{enumerate}
If $\mu(e)=\{x,y\}$ or $\mu(e)=(x,y)$, then $x$ and $y$ are called the end-vertices of $e$.
We will sometimes identify the edge $e$ with $\mu(e)$ and write $e=\mu(e)$.
Let $L(G)$ denote the set of all loops of $G$ and, given a vertex $x$ of $G$,
we denote by $G_{x}$ the subgraph of $G$ induced by all the edges
(including loops) of $G$ incident with $x$.

An \emph{edge-coloring} of $G$ with $c$ colors is a map $\lambda: E(G)\rightarrow [1,c]$, where $[1,c]$ denotes the set of all positive integers not greater than $c$.
The set $\lambda^{-1}(j)$ of all edges colored $j$ is referred to as the $j$-color class. Clearly, the color classes of $G$ partition between them $E(G)$, hence $|E(G)|=\sum_{j=1}^c|\lambda^{-1}(j)|$.

A \emph{tile} is a finite multiset of $\mathbb{Z}$ and a set of tiles is called a \emph{pot}.
Given a pot $P$, we denote by $\Sigma(P)\subset\mathbb{N}$ the set of all distinct positive integers (colors) appearing in some tile of $P$.
In the following, we define the concepts of tile and pot induced by
an edge-coloring $\lambda$.
\begin{enumerate}
\item The tile induced by $\lambda$ on
$x\in V(G)$ is the multiset
\begin{align*}
&\tau_x(\lambda)=
\{e_x\lambda(e) \mid e\in E(G_{x})\}\,\cup\,
\{\lambda(e) \mid e\in L(G_{x})\}\, \text{where}
\\
&e_x =
\begin{cases}
 1 & \text{if $e=\{x,y\}$, or $e=(x,y)$ and $x\neq y$},\\
 -1 & \text{if $e=(y,x)$},\\
 0 & \text{otherwise}.
\end{cases}
\end{align*}

\item The pot induced by $\lambda$ (on $G$) is the set $\mathcal{P}(\lambda)=\{\tau_x(\lambda)\mid x\in V(G)\}$ of distinct tiles induced by $\lambda$.
\end{enumerate}

We note that the concept of tile given above coincides with the equivalent notion of \emph{tile type} used in previous works on this topic, as in \cite{Ellis2019,Ellis2014}.
We also recall that for an edge-coloring $\lambda$ of an undirected graph $G$, the notion of tile $\tau_x(\lambda)$ on $x\in V(G)$ coincides with the notion of multi-palette of $x$ with respect to $\lambda$ (see \cite{Bonvicini2020}).

\begin{example}\label{ex1}
Let $G$ be the graph in Figure \ref{FigEx1} and $\lambda: E(G) \rightarrow [1,3]$ the associated edge-coloring.
The tiles induced by $\lambda$ are: $\tau_{x_0}(\lambda)=\{1,^2 2\}$, $\tau_{x_1}(\lambda)=\{-1,-2,3\}$,
$\tau_{x_2}(\lambda)=\{1,-2,-3\}$, $\tau_{x_3}(\lambda)=\{1,^2 -1\}$.
Hence the set $$\mathcal{P}(\lambda)=\{\{1,^2 2\},\{-1,-2,3\},\{1,-2,-3\},\{1,^2 -1\}\}$$ is the pot induced by $\lambda$ on $G$.
\end{example}

\begin{example}\label{ex2}
Let $G_1$ be the graph in Figure \ref{FigEx2} and $\lambda_1: E(G_1) \rightarrow [1,3]$ the associated edge-coloring.
Note that the pot induced by $\lambda_1$ is nothing but the pot
$\mathcal{P}(\lambda)$ of Example \ref{ex1}.
\end{example}

\begin{figure}[ht]
\begin{minipage}[b]{5.7cm}
\centering
\includegraphics[width=3.5cm]{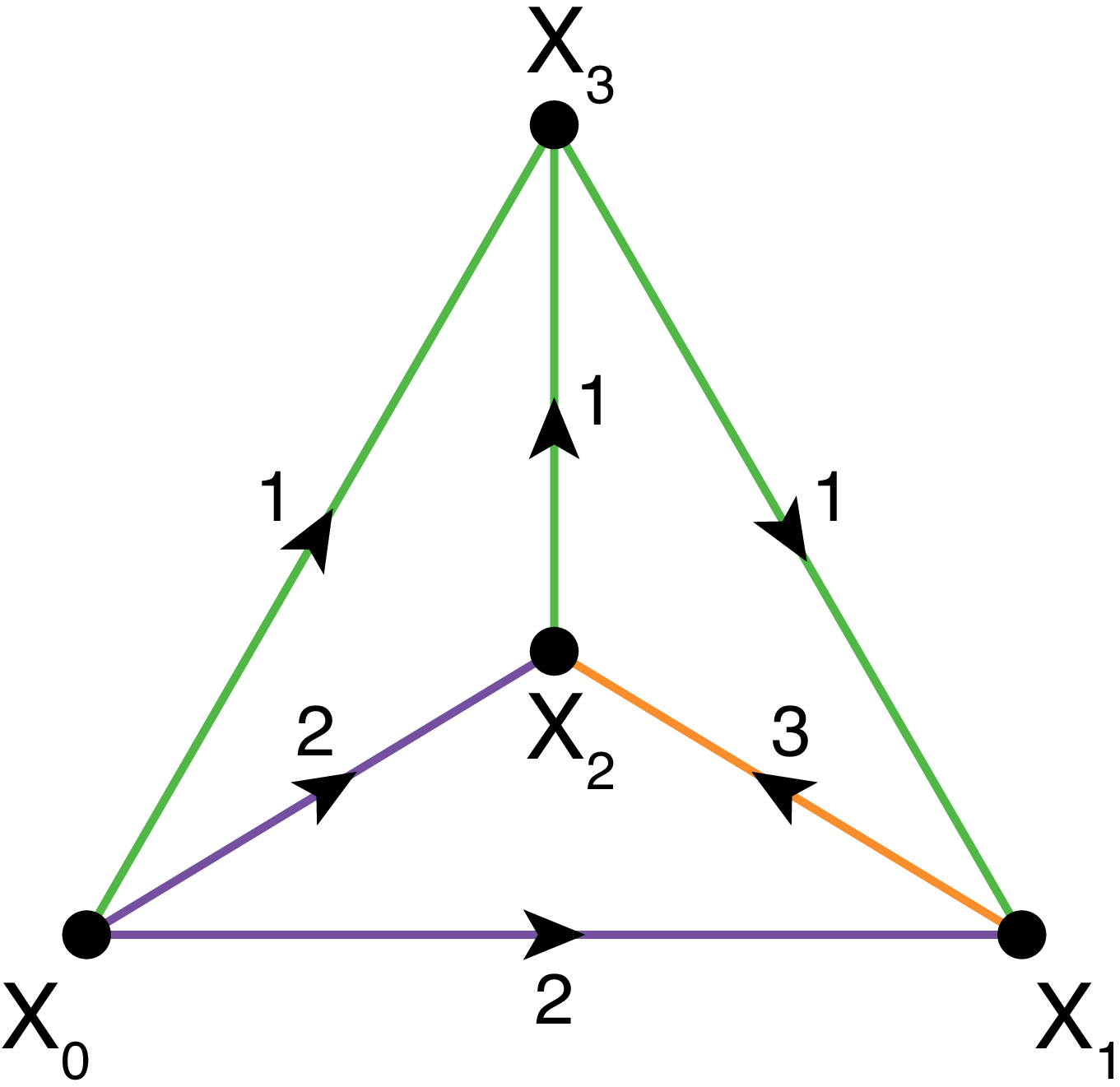}
\caption{\ }
\label{FigEx1}
\end{minipage}
\ \hspace{2mm} \hspace{3mm} \
\begin{minipage}[b]{5.7cm}
\centering
\includegraphics[width=3.5cm]{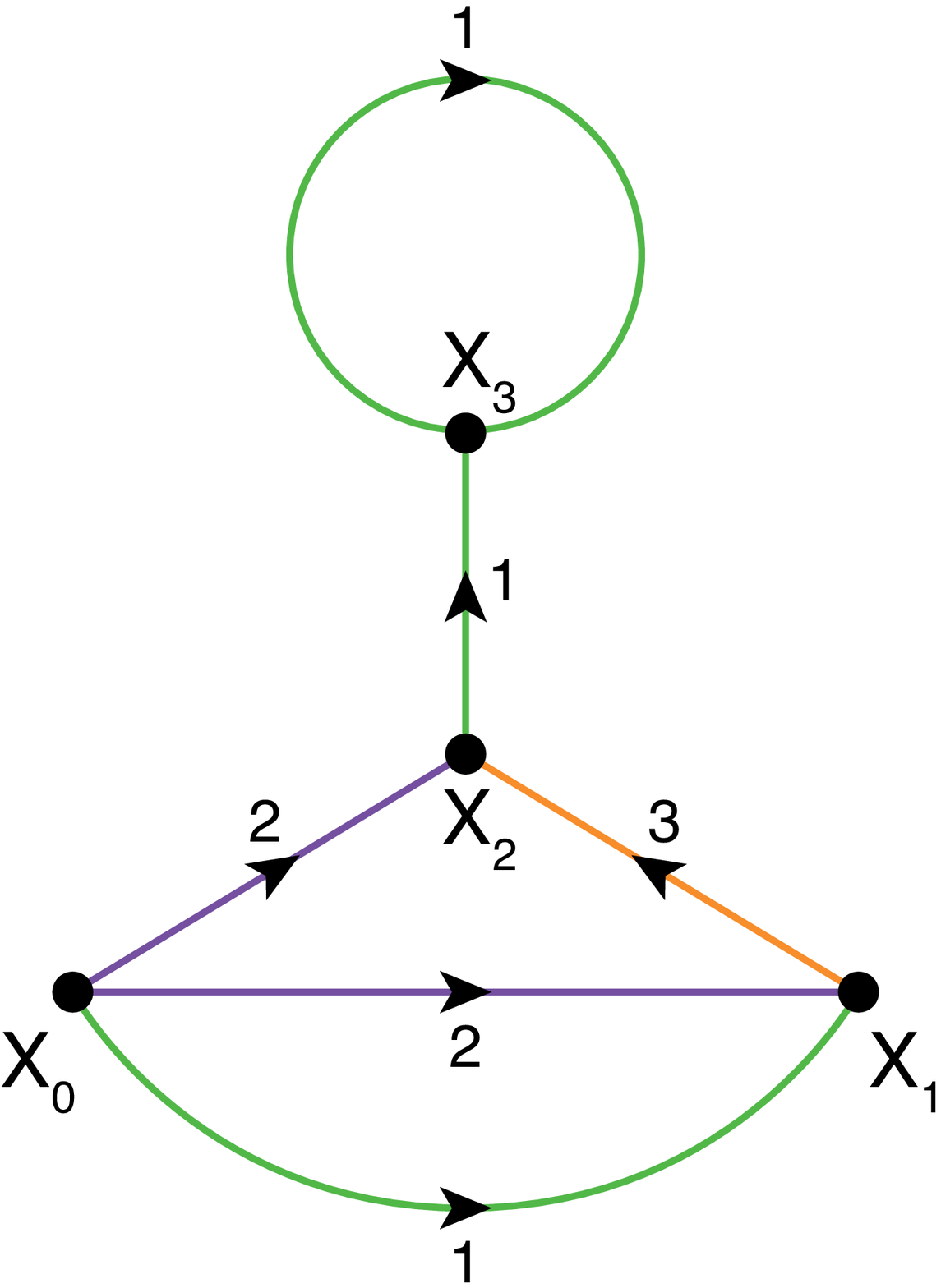}
\caption{\ }
\label{FigEx2}
\end{minipage}
\end{figure}

From now on, $G$ denotes an undirected graph.
Also, by $G^*$ we mean the set of all \emph{orientations of $G$},
that is, all directed graphs obtained from $G$ by choosing an oriention of its edges.
We denote an arbitrary graph in $G^*$ by $\dG$, and for every edge $e\in E(G)$, by $\vec{e}$ and $\cev{e}$ we mean the two orientations of $e$.

We recall here the standard terminology and notation used for some undirected graphs, see for instance \cite{W}.
The \emph{degree} of a vertex $x$ of a simple graph $G$ is the number of vertices of $G$ adjacent to $x$.
If every vertex of $G$ has degree $r$, one says that $G$ is $r$-regular. In particular, a $3$-regular graph,
is also called a \emph{cubic} graph.
A set of pairwise non-adjacent vertices of a graph $G$ is called \emph{independent}.
A graph $G$ is called \emph{bipartite} if $V(G)$ has a bipartition into two independent subsets. We denote by $C=(x_1, \ldots, x_k)$ the \emph{cycle} $C$ of length $k$ whose edges are $\{x_1, x_2\}$,
$\ldots, \{x_{k-1}, x_k\}, \{x_k, x_{k+1}=x_1\}$.
By $[x_1, \ldots, x_k]$ we denote the \emph{path} with $k$ vertices (and length $k-1$) obtained from $C$ by removing the edge $\{x_1, x_k\}$.
We recall that a graph is bipartite if and only if it contains no cycle of odd length. A \emph{star} of size $k$ (a $k$-star) is a set of $k$ edges,
none of which are loops, all incident in the same vertex.
Clearly, a $2$-star is nothing but a path of length $2$.
A \emph{matching} of a graph $G$ is a set of pairwise non incident edges of $G$, none of which are loops.

Let $\lambda: E(\dG)\rightarrow [1,c]$ be an edge-coloring
of an orientation of $G$.
The edge-coloring $\underline{\lambda}$ of $G$ induced by $\lambda$ is defined as follows:
if $\vec{e}\in E(\dG)$ is an orientation of $e\in E(G)$, then
$\underline{\lambda}(e) = \lambda(\vec{e})$.
Given two orientations $\dG_1$ and $\dG_2$ of $G$, the map $\sigma: E(\dG_1)\rightarrow E(\dG_2)$ such that $\sigma(\vec{e})\in\{\vec{e}, \cev{e}\}$ is called the
\emph{orientation map}  from $\dG_1$ to $\dG_2$.
Clearly, $\sigma^{-1}$ is an orientation map, as well.

Given a map $f$ defined on a set $X\subseteq \mathbb{Z}$,
we extend $f$ pointwise to any (family of) multisets of $X$.
For example, letting $\|z\|$ denote the absolute value of $z\in\mathbb{Z}$, and
given the tile $t=\{\,^{\mu_i}z_i\mid i=1,\ldots,u\}$ containing exactly $\mu_i$ copies of $z_i$ for each $i$, then $\|t\|=\{\,^{\mu_i}\|z_i\|\mid i\in t\}$.

The following result points out a clear connection between tiles and pots of $\lambda$ and $\underline{\lambda}$.
Its proof is straightforward and therefore left to the reader.
Note that this inequality was also obtained in~\cite{Bonvicini2020}.

\begin{lemma}\label{_lambda}
$\tau_{x}(\underline{\lambda})= \|\tau_{x}(\lambda)\|$ and
$\mathcal{P}(\underline{\lambda}) = \|\mathcal{P}(\lambda)\|$.
Therefore,
$|\mathcal{P}(\underline{\lambda})|\leq |\mathcal{P}(\lambda)|$.
\end{lemma}

\begin{definition}
 Let $G$ be an undirected graph and let $P$ be a pot.
 A \emph{realization} of $G$ through $P$ is an edge-coloring $\lambda$ of an orientation $\dG$ of $G$ such that
$\mathcal{P}(\lambda)\subseteq P$. In this case, we say that $P$ realizes $G$ and write $G\in \mathcal{O}(P)$ (meaning that $G$ is in the output of $P$).

Furthermore, we say that $P$ realizes $G$ in Scenario $i\in[1,3]$,
and write $G\in \mathcal{O}_i(P)$,
if the following conditions hold:
\begin{description}
  \item[$i=1$] $G\in \mathcal{O}(P)$.
  \item[$i=2$] $G\in \mathcal{O}(P)$ and there is no $H\in \mathcal{O}(P)$ such that $|V(H)|<|V(G)|$.
  \item[$i=3$] $G\in \mathcal{O}(P)$ and there is no $H\in \mathcal{O}(P)$ such that $H\not\simeq G$ and $|V(H)|\leq|V(G)|$.
\end{description}
If $G\in \mathcal{O}_i(P)$, we say that $P$ $i$-\emph{realizes} $G$, or there exists an $i$-\emph{realization} of $G$ through $P$.\\
\end{definition}
\noindent
The  minimum size of a pot $i$-realizing $G$ is denoted by
\[
T_i(G) = min\{|P| \mid \text{$P$ is a pot that $i$-realizes $G$}\}.
\]
The minimum number of colors used by a pot that $i$-realizes $G$ is denoted by
\[
B_i(G) = min\{|\Sigma(P)| \mid \text{$P$ is a pot that $i$-realizes $G$}\}.
\]
These parameters are called \emph{minimum number of tiles} (or, equivalently, of tile types) and \emph{minimum number of bond-edge types}, respectively. For $i\in[1,3]$, $T_i(G)$ and $B_i(G)$ have been determined for cycles, trees, complete and complete bipartite graphs in~\cite{Ellis2014}. Recent works have instead focused on Platonic solids, square and triangular lattices, and other graph classes~\cite{Almodovar21,Almodovar2021b,Almodovar2019,Griffin,Mattamira2020}.

\begin{definition}
A pot $P$ that $i$-realizes $G$ is \emph{biminimal} if
$|P| = T_i(G)$ and $|\Sigma(P)| = B_i(G)$.
\end{definition}

We now introduce the concept of isomorphic pots.

\begin{definition}
Two pots $P$ and $P'$ are \emph{isomorphic} if there exists
a bijection (isomorphism) $f:\pm\Sigma(P)\rightarrow\pm\Sigma(P')$ such that
$f(-i)=-f(i)$ and $f(P)=P'$.
\end{definition}

Clearly, isomorphic pots have the same cardinality and use the same number of colors.
Given two isomorphic pots $P$ and $P'$, we denote by $f^+:\Sigma(P)\rightarrow \Sigma(P')$ the map defined as follows: $f^+(i) = \|f(i)\|$.

\begin{definition}
Let $\lambda:E(\dG)\rightarrow\Sigma(P)$ be a realization of $G$ through $P$.
A {$\lambda$-\emph{orientation map}} is an orientation map from $\dG$ to $\dG'$
that reverses or preserves the orientation of all the edges of $\dG$ belonging to the same color class.

Given an isomorphism $f$ from $P$ to a pot $P'$, we denote by $\sigma_{f, \lambda}:E(\dG)\rightarrow E(\dG')$ the $\lambda$-\emph{orientation map}
such that $\sigma_{f, \lambda}(\vec{e})=\vec{e}$ if and only if $f(\lambda(\vec{e}))>0$, otherwise $\sigma_{f,\lambda}(\vec{e})=\cev{e}$.
\end{definition}

The following result shows that an $i$-realization of a graph through a pot is preserved by isomorphic pots.

\begin{theorem}\label{isomorphism}
Let $\lambda:E(\dG)\rightarrow [1,c]$ be an $i$-realization of $G$
and let $P=\mathcal{P}(\lambda)$.
\begin{enumerate}
  \item[(i)] Given an isomorphism $f$ from $P$ to a pot $P'$, the map $\lambda'=f^+\circ \lambda \circ \sigma$, with
  $\sigma= \sigma^{-1}_{f, \lambda}$, is an $i$-realization of $G$ through $P'$.
  \item[(ii)] Given a bijection $g:\Sigma(P)\rightarrow X\subseteq\mathbb{N}$, a $\lambda$-orientation map $\sigma: E(\dG)\rightarrow E(\dG')$ and an automorphism $\rho$ of $G$, then $\lambda' = g\circ \lambda \circ \sigma^{-1}\circ \rho$ is an $i$-realization of $G$ where $\mathcal{P}(\lambda')$ is isomorphic to $P$ and $\Sigma(\mathcal{P}(\lambda'))=X$.
\end{enumerate}
\end{theorem}
\begin{proof} We start by proving $(i)$.
Consider the edge-coloring $\lambda'=f^+\circ \lambda \circ \sigma$ of $\dG'$, where  $\sigma^{-1} = \sigma_{f, \lambda}:E(\dG)\rightarrow E(\dG')$ is the orientation map defined as follows: $\sigma_{f,\lambda}(\vec{e})=\vec{e}$ if $f(\lambda(\vec{e}))>0$, otherwise $\sigma_{f,\lambda}(\vec{e})=\cev{e}$.

Since the maps $f, \lambda$ and $\sigma$ are all surjective, we have that $\lambda'(E(\dG')) = \Sigma(P')$. We are going to show that $\tau_x(\lambda') = f(\tau_x(\lambda))$ for every $x\in V(G)$, which implies that $\mathcal{P}(\lambda')=f(P) = P'$.
It is enough to notice that for every non-loop $\vec{e}\in E(G)$ the following property holds: if $f(\lambda(\vec{e}))>0$,
then $\sigma^{-1}(\vec{e})_x=\vec{e}_x$ and $f^+(\lambda(\vec{e})) = f(\lambda(\vec{e}))$, otherwise  $\sigma^{-1}(\vec{e})_x=-\vec{e}_x$
and $f^+(\lambda(\vec{e})) = -f(\lambda(\vec{e}))$. This implies that
\begin{equation}\label{cond}
\sigma^{-1}(\vec{e})_x\cdot f^+(\lambda(\vec{e})) = \vec{e}_x f(\lambda(\vec{e}))
= f(\vec{e}_x\lambda(\vec{e})),\; \text{for every $\vec{e}\in E(G)\setminus L(G)$}
\end{equation}
and recalling  that $\sigma(E(\dG'_x)) = E(\dG_x)$, we have that
\begin{align*}
  \tau_x(\lambda') &=
   \{\vec{e}_x\cdot \lambda'(\vec{e}) \mid \vec{e}\in E(\dG'_{x})\}
   \cup\, \{\lambda'(\vec{e}) \mid \vec{e}\in L(\dG'_{x})\}\\
&= \{\vec{e}_x\cdot f^+(\lambda(\sigma(\vec{e}))) \mid \vec{e}\in E(\dG'_{x})\}
   \cup\, \{f^+(\lambda(\sigma(\vec{e}))) \mid \vec{e}\in L(\dG'_{x})\}\\
&= \{\sigma^{-1}(\vec{e})_x\cdot f^+(\lambda(\vec{e})) \mid \vec{e}\in E(\dG_{x})\}
   \cup\, \{f^+(\lambda(\vec{e})) \mid \vec{e}\in L(\dG_{x})\}\\
&\stackrel{\eqref{cond}}{=}\{f(\vec{e}_x\cdot\lambda(\vec{e})) \mid \vec{e}\in E(\dG_{x})\setminus L(\dG_{x})\} \cup\, \{\pm f^+(\lambda(\vec{e})) \mid \vec{e}\in L(\dG_{x})\}\\
& = \{f(\vec{e}_x\cdot\lambda(\vec{e})) \mid \vec{e}\in E(\dG_{x})\setminus L(\dG_{x})\} \cup\, \{f(\lambda(\vec{e})), f(-\lambda(\vec{e})) \mid \vec{e}\in L(\dG_{x})\}\\
& = \{f(\vec{e}_x\cdot\lambda(\vec{e})) \mid \vec{e}\in E(\dG_{x})\} \cup\, \{f(\lambda(\vec{e})) \mid \vec{e}\in L(\dG_{x})\}
= f(\tau_x(\lambda)),
\end{align*}

for every $x\in V(G)$. Therefore, $\lambda'$ is a realization of $G$ through $P'$.

Assuming now that $\lambda$ is an $i$-realization of $G$ ($i=2,3$),
it is left to show that $P'$ $i$-realizes $G$. For a contradiction,
assume that $P'$ realizes a graph $H$ non isomorphic to $G$ whose order is either smaller than $|V(G)|$ ($i=2,3$) or equal to $|V(G)|$ ($i=3$). Then, by the first part of the proof it follows that $P=\mathcal{P}(\lambda)$ realizes $H$ (since $P$ and $P'$ are isomorphic), contradicting the assumption that $\lambda$ is an $i$-realization of~$G$.

It is left to prove $(ii)$. Considering that
$\tau_x(\lambda') = \tau_{\rho(x)}(\lambda'\circ \rho^{-1})$ for each $x\in V(G)$, it is enough to prove the assertion when $\rho$ is the identity, that is, $\lambda' = g\circ \lambda \circ \sigma^{-1}$. It is not difficult
to check that $\lambda'$ is a realization of $G$ through $P'=\mathcal{P}(\lambda')$,
where $\Sigma(P')=X$.
Now let $f:\pm \Sigma(P) \rightarrow \pm \Sigma(P')$ be the map such that
$f(-i)=-f(i)$, and
\[
  f(i)=
  \begin{cases}
    g(i) & \text{if $\sigma$ fixes the orientation of all edges of the $i$-color class of $\lambda$}, \\
    -g(i) & \text{otherwise},
  \end{cases}
\]
for each $i\in \Sigma(P)$. One can check that $f(P)=P'$, hence
$P$ and $P'$ are isomorphic pots. By part $(i)$, we have that $P'$ $i$-realizes $G$, therefore $\lambda'$ is an $i$-realization of~$G$.
\end{proof}

\begin{remark}\label{isomorphism:rem}
We point out that Theorem \ref{isomorphism}.(ii) states that starting from a
realization $\lambda:E(\dG)\rightarrow \Sigma(P)$ of $G$,
if we permute the colors in $\Sigma(P)$,
switch the orientation of all edges of some color class of $\lambda$,
or apply an automorphism of the graph $G$,
we obtain another realization of $G$ whose pot is isomorphic to $P$.
\end{remark}

\section{Known results}\label{sec:res}
In this section, we recall some known facts on realizations of undirected graphs.
The following result is straightforward.

\begin{remark}\label{tilestructure}
Let $\lambda:E(\dG)\rightarrow [1,c]$ be a realization of an undirected graph $G$. We note that:
\begin{enumerate}
\item[(i)] If $G$ has a loop then there exists a tile $t\in \mathcal{P}(\lambda)$ such that $|-t\cap t|\geq 2$.
\item[(ii)] If $G$ has a multiple edge then there exist two tiles $t_i,t_j\in \mathcal{P}(\lambda)$ such that $|-t_i \cap t_j|\geq 2$.
\end{enumerate}
\end{remark}

We refer the reader to Example \ref{ex2} for an application of Remark \ref{tilestructure}.

\begin{remark}\cite[Section 4]{Almodovar21}\label{incidentedges}
Let $\lambda:E(\dG)\rightarrow [1,c]$ be a $3$-realization of an undirected loopless graph $G$. We note that if $\vv{e_1},\vv{e_2}\in E(\dG_x)$ and
$\lambda(\vv{e_1})=\lambda(\vv{e_2})$, then $x$ is either the tail or the head of both $\vv{e_1}$ and $\vv{e_2}$. In other words, the edges of $\dG_x$ receiving the same color from $\lambda$ must be all outgoing from or all ingoing in $x$.
\end{remark}

Now we recall a very important class of graphs, the so called \emph{Cayley graphs}.
Let $\Gamma$ be an additive group (not necessarily abelian) and let $\Omega \subseteq \Gamma \setminus\{0\}$ such that for every $\omega \in \Omega$, also
$-\omega \in \Omega$. The Cayley graph on $\Gamma$ with connection set $\Omega$ is the simple graph $G$ defined as follows:
$V(G)=\Gamma$ and $\{x,y\}\in E(G) \Leftrightarrow x-y \in \Omega$. We will denote this graph by $Cay[\Gamma:\Omega]$.

 The cube graph $Q$ can be seen as a Cayley graph on $\mathbb{Z}_2^3$. To ease notation, each triple $(i,j,k)\in \mathbb{Z}_2^3$ will be denoted by $ijk$,
 and let  $Q=Cay[\mathbb{Z}_2^3:S]$, where $S=\{100, 010, 001\}$.
Given an edge $\{x,y\}$ of $Q$ and an element $g\in \mathbb{Z}_2^3$,
 we set
 $\{x,y\}+g=\{x+g, y+g\}$.
Note that for every $\overrightarrow{Q}\in Q^*$, a tile of $\overrightarrow{Q}$ (resp. $Q$) is a multiset of
$\mathbb{Z}$ of size $3$.

Since the cube is a $3$-regular graph, from \cite{Ellis2014} it follows that $B_1(Q)=1$ and $T_1(Q)=2$.
Furthermore, from \cite{Almodovar21} we know that $B_2(Q)=2$ and $T_2(Q)=3$ and these values are achieved simultaneously by the same pot.
In \cite{Almodovar21}, it is also shown that a $3$-realization of $Q$ uses at least $5$ colors ($B_3(Q)\geq 5$) and determines pots of size at least six ($T_3(Q)\geq 6$). They also construct two pots $P$ and $P'$ that
$3$-realizes $Q$ and reach the minimum number of colors or of distinct tiles, respectively; more precisely,
$|\Sigma(P)|=5, |P|=8$, while
$|\Sigma(P')|=|P'|=6$. Therefore, we have the following.
\begin{lemma}\cite{Almodovar21}\label{BT3cube}
  $B_3(Q)=5$ and $T_3(Q) = 6$.
\end{lemma}
In \cite{Almodovar21}, the authors leave open the problem of determining whether there exists a biminimal $3$-realization of $Q$.

The following lemma is a restatement of \cite[Lemma 4.4]{Almodovar21} in terms of color classes.

\begin{lemma}\cite[Lemma 4.4]{Almodovar21}
\label{Tom}
Let $\lambda:E(\vv{Q})\rightarrow [1,c]$ be a $3$-realization of $Q$
and let $\mathcal{C}$ be the set of color classes of $\underline{\lambda}$.
\begin{enumerate}
\item[(i)] A color class $G\in \mathcal{C}$
is either a star (with $1$, $2$ or $3$ edges) or a matching $\{e, e+111\}$ for some $e\in E(Q)$.
\item[(ii)]
$\mathcal{C}$ contains $n$ $3$-stars (resp. $2$-stars)
if and only if
$\mathcal{P}(\lambda)$ contains  $n$ monochromatic (bichromatic) tiles.
\end{enumerate}
\end{lemma}

%

\section{The structure of a biminimal pot that $3$-realizes the cube}\label{sec:pots}
In this section, we provide two biminimal pots, $P_1$ and $P_2$, that realize the cube in Scenario 3 (Propositions \ref{prop:pot1} and \ref{prop:pot2}). We then show that these pots are unique up to isomorphisms (Theorem \ref{unique}).

\begin{proposition}\label{prop:pot1}
The pot
$$P_1= \big\{ \{^31\}, \{^32\}, \{-2, \,^2{-3}\}, \{-1, \,^2{-4}\}, \{-1,3,-5\}, \{-2,4,5\}\big\}$$
is a biminimal pot that $3$-realizes the cube.
\end{proposition}
\begin{proof}
Let $\lambda:E(\dQ)\rightarrow [1,5]$ be the edge-coloring of $Q$ inducing the pot $P_1$, that is $P_1=\mathcal{P}(\lambda)$.
In Figure \ref{prop1} it is shown that $P_1$ realizes the cube and the following tabular shows the tiles:
\[
\begin{tabular}{r|c|c|c|c|c|c}
$x$ & 000 & 111 & 100 & 011 & $001,010$ & $101,110$\\ \hline
$\tau_x(\lambda)$ & $\{^3 1\}$ & $\{^3 2\}$ & $\{-1, ^2 -4\}$ & $\{-2, ^2 -3\}$  & $\{-1, 3, -5\}$ &  $\{-2, 4, 5\}$
\end{tabular}
\]
First of all note that, by Remark \ref{tilestructure}, if $G$ is a graph such that $G \in \mathcal{O}(P_1)$,
then it is a simple graph.
Now we prove that $G$ has at least $8$ vertices, to do this it is useful to analyze solutions to the equations
that the tiles in $P_1$ must satisfy. For $i=1,\ldots,6$, denote by $R_i$ the number of times a tile $t_i$
appears in the realization, or, equivalently, the number of vertices in $\dG$ receiving the tile $t_i$,
that is
$R_i=|\{x\in V(G)\mid \tau_x(\lambda)=t_i\}|$. Clearly, $R_1+\cdots +R_6=|V(G)|$.
Also, as shown in \cite{Ellis2014},  the $R_i$'s have to satisfy the following homogeneous system of linear equations:
\begin{eqnarray*}
 3R_1-R_3-R_5 &=& 0 \\
  3R_2-R_4-R_6 &=& 0 \\
  -2R_3+R_6 &=& 0 \\
  -2R_4+R_5 &=& 0 \\
  -R_5+R_6 &=& 0
\end{eqnarray*}
%
Therefore,
$(R_1, \ldots, R_6)=r(1,1,1,1,2,2)$  and $|V(G)|=8r$ for some $r\geq 1$.
Hence we get $|V(G)|\geq 8$, and when $r=1$, we have $(R_1, \ldots, R_6)=(1,1,1,1,2,2)$.
So we focus on simple cubic graphs with $8$ vertices and it is well known, see for instance \cite{GraphsHB}, that there are exactly five graphs with these properties.
Let $H$ be a cubic graph such that $V(H)=\{v_1,v_2,\ldots,v_8\}$ and $H \in \mathcal{O}(P_1)$.
Set
\[
\begin{tabular}{r|c|c|c|c|c|c}
$v_i$ & $v_1$ & $v_2$ & $v_3$ & $v_4$ & $v_5,v_6$ & $v_7,v_8$\\ \hline
$\tau_{v_i}(\lambda)$ & $\{^3 1\}$ & $\{^3 2\}$ & $\{-1, ^2 -4\}$ & $\{-2, ^2 -3\}$  & $\{-1, 3, -5\}$ &  $\{-2, 4, 5\}$
\end{tabular}
\]
Note that if two distinct vertices $x,y$ are adjacent in $H$ then there exists $j \in [1,5]$ such that
$j \in \tau_x(\lambda)$ and $-j \in \tau_y(\lambda)$.
Hence one can directly check that $S_1=\{v_1,v_4,v_7,v_8\}$ and $S_2=\{v_2,v_3,v_5,v_6\}$ are two maximal disjoint independent sets of vertices of $H$,
such that $S_1 \cup S_2 =V(H)$.
In other words, the graph $H$ is a bipartite graph with parts $S_1$ and $S_2$. Since among the five cubic graphs of order $8$ the unique bipartite graph is the cube, see \cite{GraphsHB},
the graph $H$ is the cube.
\end{proof}

\begin{figure}[ht]
\begin{minipage}[b]{5.7cm}
\centering
\includegraphics[width=3.5cm]{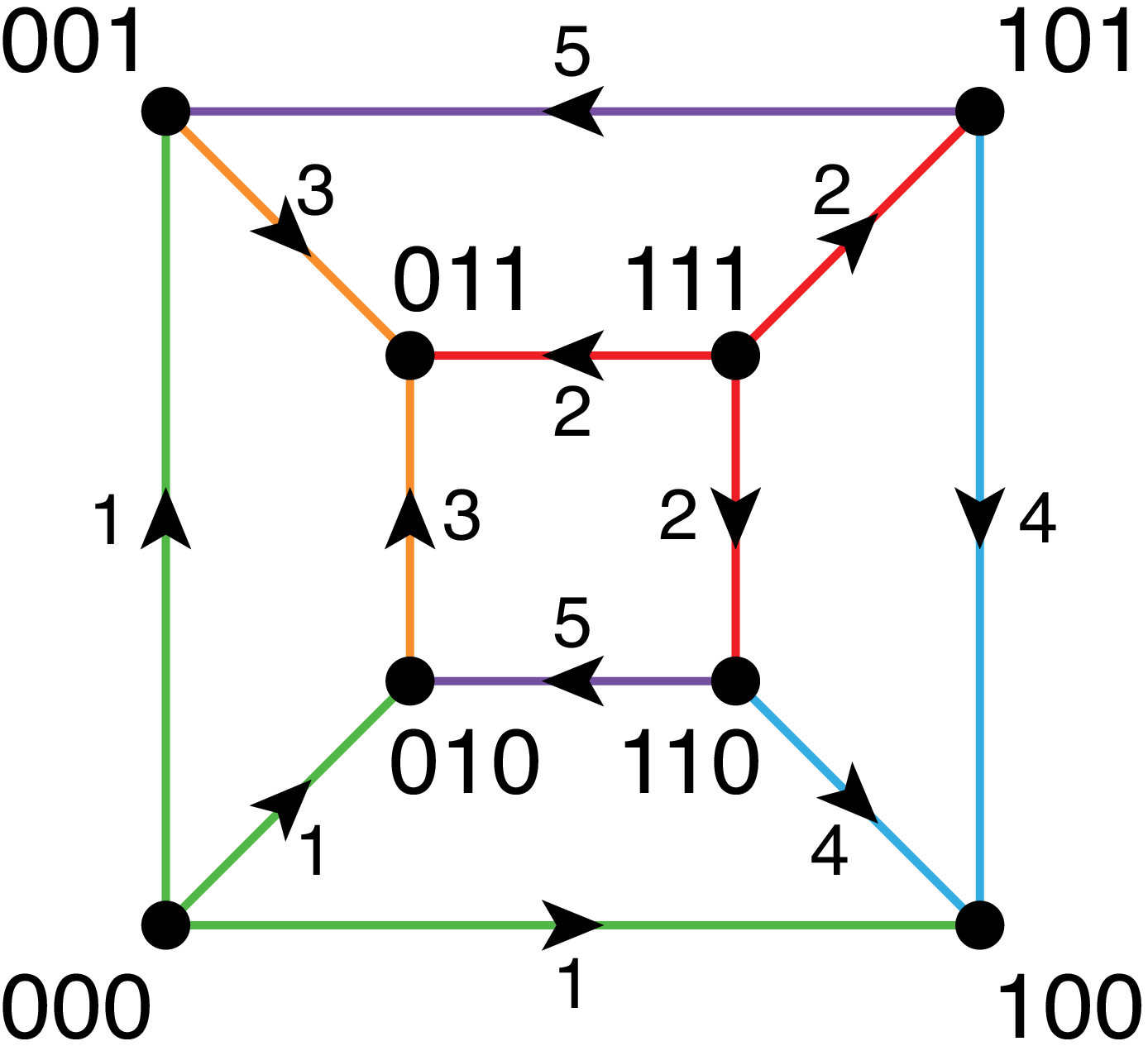}
\caption{\ }
\label{prop1}
\end{minipage}
\ \hspace{2mm} \hspace{3mm} \
\begin{minipage}[b]{5.7cm}
\centering
\includegraphics[width=3.5cm]{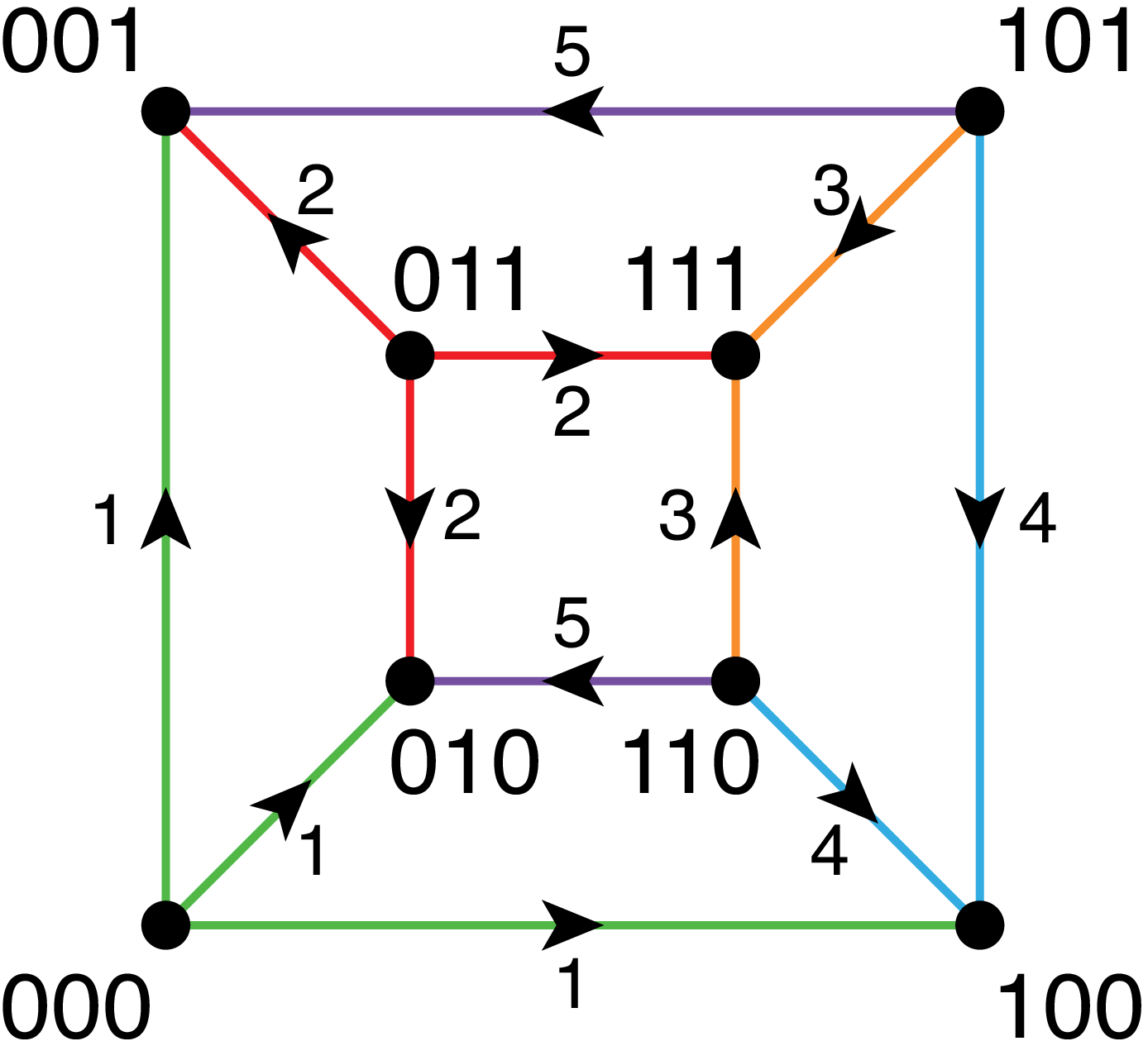}
\caption{\ }
\label{prop2}
\end{minipage}
\end{figure}

\begin{proposition}\label{prop:pot2}
The pot
$$P_2= \big\{ \{^31\}, \{^32\}, \{-2, \,^2{-3}\}, \{-1, \,^2{-4}\},  \{3,4,5\}, \{-1,-2,-5\}\big\}$$
is a biminimal pot that $3$-realizes the cube.
\end{proposition}
\begin{proof}
Let $\lambda:E(\dQ)\rightarrow [1,5]$ be the edge-coloring of $Q$ inducing the pot $P_2$, that is $P_2=\mathcal{P}(\lambda)$,
see  Figure \ref{prop2}.
The following tabular shows the tiles:
\[
\begin{tabular}{r|c|c|c|c|c|c}
$x$ & 000 & 011 & 111 & 100 & $101,110$ & $001,010$\\ \hline
$\tau_x(\lambda)$ & $\{^3 1\}$ & $\{^3 2\}$ & $\{-2, ^2 -3\}$ & $\{-1, ^2 -4\}$  & $\{3,4,5\}$ &  $\{-1,-2,-5\}$
\end{tabular}
\]
Then the result can be proved with the same reasoning of the proof of Proposition \ref{prop:pot1}.
\end{proof}

\begin{remark}
Note that $P_1$ e $P_2$ are non-isomorphic pots. By contradiction, assume there exists an isomorphism $f$ between $P_1$ and $P_2$. Recall that $f(-i) = -f(i)$ for all $i$ in $\pm\Sigma(P_1)=\pm\Sigma(P_2) = \pm[1,5]$. Then $f$ either fixes colors $1$, $2$, $3$, and $4$, or $f$ swaps $1$ with $2$ and $3$ with $4$. Therefore, in both cases, $f([1,4])=[1,4]$. It follows that there is no tile in $P_1$ whose image is the tile $\{3,4,5\}\in P_2$.
\end{remark}


\begin{lemma}\label{monochromatic}
If $P$ is a pot that $3$-realizes $Q$ and $|\Sigma(P)|=5$,
then $P$ has at least two monochromatic tiles.
\end{lemma}
\begin{proof}
Let $\lambda:E(\dQ)\rightarrow [1,5]$ be a $3$-realization of $Q$ and set $P=\mathcal{P}(\lambda)$. By Lemma \ref{Tom}.$(i)$, the color classes induced by $\lambda$ on $Q$ have size at most $3$. Denoting by $c_1$ the number of those of size 3, we have that
\[
12=|E(Q)|=\sum_{i=1}^5 |\lambda^{-1}(i)|
\leq 3c_1 + 2(5-c_1) = c_1 + 10.
\]
By Lemma \ref{Tom}$.(ii)$, it follows that $P$ has $c_1\geq 2$ monochromatic tiles.
\end{proof}

We now prove that the pots in Proposition \ref{prop:pot1} and Proposition \ref{prop:pot2} are the unique (up to isomorphisms) biminimal pots that $3$-realize the cube.

\begin{theorem}\label{unique}
A biminimal pot that $3$-realizes the cube is necessarily isomorphic to one of the following:
\begin{enumerate}
  \item[(i)]
    $P_1=\big\{ \{^31\}, \{^32\}, \{-2, \,^2{-3}\}, \{-1, \,^2{-4}\},
            \{-1,3,-5\}, \{-2,4,5\}\big\}$;

  \item[(ii)] $P_2=\big\{ \{^31\}, \{^32\}, \{-2, \,^2{-3}\}, \{-1, \,^2{-4}\},
                        \{3,4,5\}, \{-1,-2,-5\}\big\}$.
\end{enumerate}
\end{theorem}
\begin{proof}
Let $\lambda: E(\dQ)\rightarrow [1,c]$ be a $3$-realization of $Q$
and set $P=\mathcal{P}(\lambda)$. We assume that $P$ is biminimal, that is, $c=5$ and $|P|=6$. By Theorem \ref{isomorphism} and Remark \ref{isomorphism:rem},
it is enough to show that
up to a permutation of the colors $[1,c]$,
up to an automorphism of $Q$, and
by switching (if necessary) the directions of all edges of some color classes of $\lambda$,
we obtain a $3$-realization of $Q$ whose pot is either $1$ or $2$.

Let $Q(i)$ be the $i$-color class determined by $\underline{\lambda}$ on $Q$, for $1\leq i \leq 5$, and by Lemma \ref{monochromatic} we have that $P$ contains at least two monochromatic tiles, say  $t_1$ and $t_2$. Up to a permutation of the colors, we may assume that $\|t_j\|=\{^3 j\}$, for $j=1,2$.
By Lemmas \ref{Tom} and \ref{monochromatic}, there exist at least two colors classes isomorphic to a $3$-star: up to a color permutation, we assume that
$\underline{\lambda} (Q(i))=\|t_i\|$. It is not restrictive to assume that $t_1$ is the tile of vertex $000$, and
in view of the simmetries of the cube, we may assume without loss of generality that $t_2$ is the tile of either $111$ or $011$.

\begin{description}[style=unboxed,leftmargin=0cm]
  \item[Case 1]$t_2=\tau_{111}(\lambda)$. Then, as shown in Figure \ref{Case1}, $Q\setminus\{Q_{000}\,\cup\, Q_{111}\}$ is the $6$-cycle
\[C=(010, 011,001,101,100,110).\]
Note that $Q(3), Q(4), Q(5)$ can be seen as the color classes induced by
$\underline{\lambda}$ on $C$. By Lemma \ref{Tom}, we necessarily have that each $Q(i)$, $i=3,4,5$, has size exactly $2$, hence each $Q(i)$ is either a $2$-star or a matching $\{e, e+111\}$ for some $e\in E(C)$.
\begin{figure}[ht]
\begin{minipage}[b]{5.7cm}
\centering
\includegraphics[width=3.5cm]{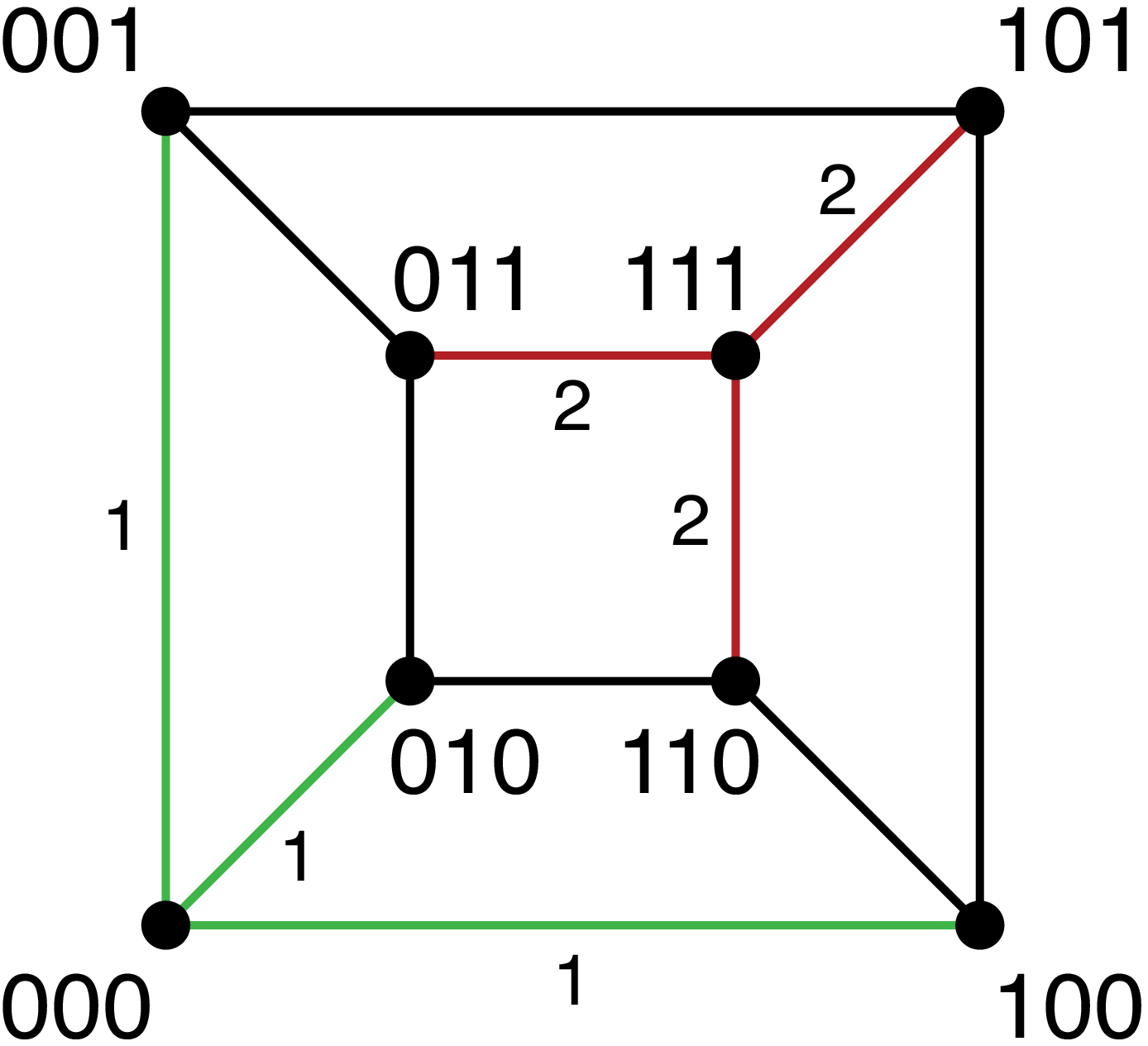}
\caption{\ }
\label{Case1}
\end{minipage}
\ \hspace{2mm} \hspace{3mm} \
\begin{minipage}[b]{5.7cm}
\centering
\includegraphics[width=3.5cm]{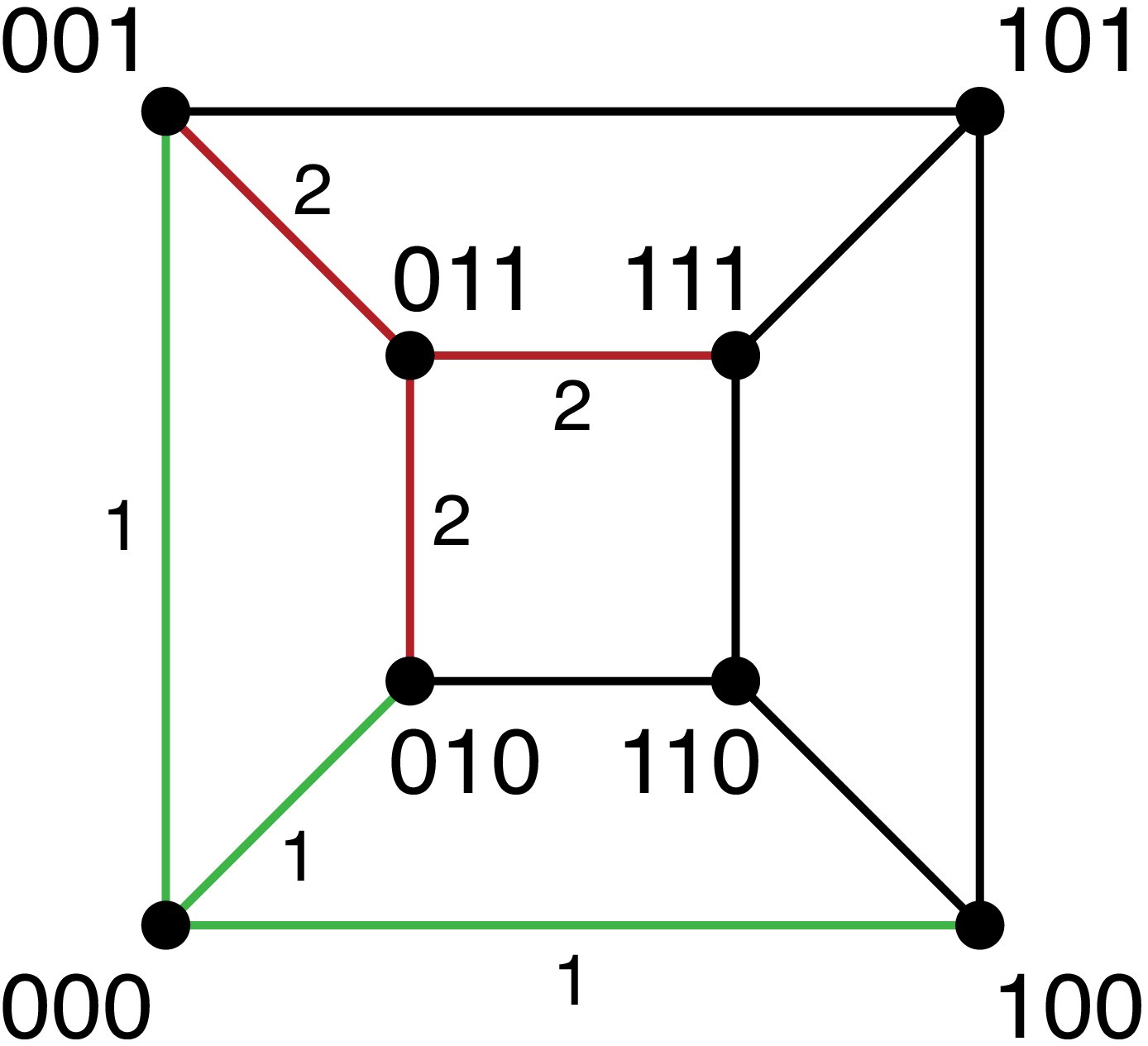}
\caption{\ }
\label{Case2}
\end{minipage}
\end{figure}
  \begin{enumerate}[style=unboxed,leftmargin=0cm]
\item[] Case 1.1:
    $Q(3), Q(4), Q(5)$ are matchings. Up to a permutation of the colors, we may assume that
$Q(3)=\{\{010, 011\}, \{101, 100\}\}$,
$Q(4)=\{\{011,001\}, \{100, 110\}\}$,
$Q(5)=\{\{001,101\}, \{110,010\}\}$.
The edge-coloring $\underline{\lambda}$ is then completely determined by its color classes, and one can check that $|\mathcal{P}(\underline{\lambda})|\geq 7$. By Lemma
\ref{_lambda}, it follows that $|\mathcal{P}(\lambda)|\geq 7$ contradicting the assumption.

\item[] Case 1.2: $Q(3), Q(4), Q(5)$ are $2$-stars. Up to a permutation of the colors, we may assume that
$Q(3)=[010, 011,001]$,
$Q(4)=[001,101,100]$,
$Q(5)=[100,110,010]$.
As before, one can check that $|\mathcal{P}(\underline{\lambda})|\geq 7$.
By Lemma \ref{_lambda}, it follows that $|\mathcal{P}(\lambda)|\geq 7$ contradicting the assumption.

\item[] Case 1.3: $Q(3), Q(4)$ are $2$-stars, and $Q(5)$ is a matching.
Up to a permutation of the colors, we may assume that
$Q(3)=[010,011,001]$,
$Q(4)=[101,100,110]$,
$Q(5)=\{\{001,101\}, \{110,010\}\}$.
One can check that
\[
\begin{tabular}{r|c|c|c|c|c|c}
$x$ & 000 & 111 & 011 & 100 & $001,010$ & $101,110$\\ \hline
$\tau_x(\underline{\lambda})$ &
    $\{^31\}$ & $\{^32\}$ & $\{2, ^23\}$ & $\{1, ^24\}$ & $\{1, 3, 5\}$
    & $\{2, 4, 5\}$
\end{tabular}
\]
In view of Remark \ref{incidentedges}, by switching if necessary the directions of all edges of $\dQ$ with the same color,
we can assume that the edges of $\dQ_{000}$ (resp. $\dQ_{111}$) are all outgoing from $000$ (resp. $111$): hence
\[
\text{$\tau_{000}(\lambda)=\{^31\}$ and $\tau_{111}(\lambda)=\{^32\}$}.
\]
We can also assume that
the two edges of $\dQ$ colored 3 (resp. $4$) are ingoing in $011$ (resp. $100$), while the two edges colored $5$ are outgoing from 101 and 110.
Therefore, $\lambda$ is the realization of $Q$ shown in Figure \ref{prop1}, and
\[
\text{
$\tau_{011}(\lambda)=\{-2, \,^2{-3}\}$,
$\tau_{100}(\lambda)=\{-1, \,^2-4\}$,
}
\]
\[
\text{
$\tau_{001}(\lambda)=\tau_{010}(\lambda)=\{-1,3,-5\}$
and
$\tau_{101}(\lambda)=\tau_{110}(\lambda)=\{-2,4,5\}$.
}
\]
  \end{enumerate}


  \item[Case 2] $t_2=\tau_{011}(\lambda)$. Then, as shown in Figure \ref{Case2}, $Q\setminus\{Q_{000}\,\cup\, Q_{111}\}$ is the graph $G=C \,\cup\, \{e, e+111\}$ consisting of a $4$-cycle $C=(110, 111,101,100)$ with two pendant edges $e=\{110, 010\}$ and
  $e+111=\{001,101\}$.
Note that $Q(3), Q(4), Q(5)$ can be seen as the color classes induced by
$\underline{\lambda}$ on $C$. By Lemma \ref{Tom}, we necessarily have that either some $Q(i)$ has size $3$, or each $Q(i)$ (has size $2$, hence) is either a $2$-star or the matching $\{e, e+111\}$.

  \begin{enumerate}[style=unboxed,leftmargin=0cm]
     \item[] Case 2.1: some $Q(i)$, $3\leq i\leq 5$, is a $3$-star. Up to a permutation of the colors, we may assume that $Q(3)$ is a $3$-star. Also, the center of $Q(3)$ is either $110$ or $101$.
%
     In both cases, one can check that
     $|\mathcal{P}(\underline{\lambda})|\geq 7$, hence (by Lemma
\ref{_lambda}) $|\mathcal{P}(\lambda)|\geq 7$ contradicting the assumption.

     \item[] Case 2.2: $Q(3), Q(4)$ and $Q(5)$ are $2$-stars. Then, up to a permutation of the colors, we have that either
     \begin{align*}
       Q(3)=[010, 110,111],\; Q(4)=[110,100,101],\; Q(5)=[111,101,001], \text{or}\\
       Q(3)=[010, 110,100],\; Q(4)=[001,101,100],\; Q(5)=[110,111,101].
     \end{align*}
In both cases, one can check that $|\mathcal{P}(\underline{\lambda})|\geq 7$,
hence (by Lemma \ref{_lambda}) $|\mathcal{P}(\lambda)|\geq 7$ contradicting the assumption.

    \item[] Case 2.3: $\{e, e+111\}$ is a color class of $\underline{\lambda}$.
    Up to a permutation of the colors, we may assume that $Q(5) = \{e, e+111\}$.
    Therefore, both $Q(3)$ and $Q(4)$ are $2$-stars of the $4$-cycle $C=(110, 111,101,100)$. If $Q(3)=[100,110,111]$ and $Q(4)=[111,101, 100]$, one can check that
    $|\mathcal{P}(\underline{\lambda})|\geq 7$,
    hence (by Lemma \ref{_lambda}) $|\mathcal{P}(\lambda)|\geq 7$
    contradicting the assumption. Therefore, $Q(3)=[110,111,101]$, $Q(4)=[101,100, 110]$, and
\[
\begin{tabular}{r|c|c|c|c|c|c}
$x$ & 000 & 011 & 111 & 100 & $001,010$ & $101,110$\\ \hline
$\tau_x(\underline{\lambda})$ &
    $\{^31\}$ & $\{^32\}$ & $\{2, ^23\}$ & $\{1, ^24\}$ & $\{1, 2, 5\}$
    & $\{3, 4, 5\}$
\end{tabular}
\]
In view of Remark \ref{incidentedges}, by switching if necessary the directions of all edges of $\dQ$ with the same color,
we can assume that the edges of $\dQ_{000}$ (resp. $\dQ_{011}$) are all outgoing from $000$ (resp. $011$): hence $\tau_{000}(\lambda)=\{^31\}$ and $\tau_{011}(\lambda)=\{^32\}$.
We can also assume that
the two edges of $\dG$ colored $3$ (resp. $4$) are ingoing in $111$ (resp. $100$), while the two edges colored $5$ are outgoing from $101$ and $110$.
Therefore, $\lambda$ is the realization of $Q$ shown in Figure \ref{prop2}, and
\[
\text{
$\tau_{111}(\lambda)=\{-2, \,^2{-3}\}$,
$\tau_{100}(\lambda)=\{-1, \,^2{-4}\}$,
}
\]
\[
\text{
$\tau_{001}(\lambda)=\tau_{010}(\lambda)=\{-1,-2,-5\}$
and
$\tau_{101}(\lambda)=\tau_{110}(\lambda)=\{3,4,5\}$.
}
\]
  \end{enumerate}
\end{description}

\end{proof}


\section{Acknowledgments}
The research of M.M. Ferrari was supported by the NSF-Simons Southeast Center for Mathematics and Biology (SCMB) through the grants National Science Foundation DMS1764406 and Simons Foundation/SFARI 594594.
 A. Pasotti and T. Traetta received support from INdAM-GNSAGA.

The research was also supported by the Italian Minister of University and Research through the project named ``Fondi per attivit\`a a carattere internazionale'' 2019\_INTER\_DICATAM\_TRAETTA.

\end{document}